\documentclass[12pt,letterpaper]{amsart}
\usepackage[letterpaper,hmargin=1.25in,vmargin=1.5in]{geometry}
\usepackage{latexsym, amssymb, amsmath}

\newtheorem{theorem}{Theorem}[section]
\newtheorem{lemma}[theorem]{Lemma}
\newtheorem{corollary}[theorem]{Corollary}

\theoremstyle{definition}

\makeatletter
    
    \@addtoreset{equation}{section}
  \makeatother

\newcommand{\Ric}{{\rm Ric}}

\begin{document}

\title[Einstein-type K\"ahler manifolds with $\alpha=0$]
{Classification of gradient Einstein-type K\"ahler manifolds with $\alpha=0$}

\author{Shun Maeta}
\address{Department of Mathematics, Chiba University, 1-33, Yayoicho, Inage, Chiba, 263-8522, Japan.}
\curraddr{}
\email{shun.maeta@faculty.gs.chiba-u.jp~{\em or}~shun.maeta@gmail.com}
\thanks{The author is partially supported by the Grant-in-Aid for Scientific Research (C), No.23K03107, Japan Society for the Promotion of Science.}

\date{}

\dedicatory{}


\commby{}

\begin{abstract}
Thanks to the ambitious project initiated by Catino, Mastrolia, Monticelli and Rigoli, which aims to provide a unified viewpoint for various geometric solitons, many classes, including Ricci solitons, Yamabe solitons, $k$-Yamabe solitons, quasi-Yamabe solitons, and conformal solitons, can now be studied under a unified framework known as Einstein-type manifolds. 
Einstein-type manifolds are characterized by four constants, denoted by $\alpha, \beta, \mu$ and $\rho$.
In this paper, we completely classify all non-trivial, complete gradient Einstein-type K\"ahler manifolds with $\alpha = 0$. As a corollary, rotational symmetry for many classes is obtained. In particular, we show that any non-trivial complete gradient quasi-Yamabe soliton on K\"ahler manifolds is rotationally symmetric.
\end{abstract}

\maketitle

\bibliographystyle{amsalpha}


\section{Introduction}\label{intro}

Self-similar solutions to geometric flows are a fundamental and important topic in the study of Riemannian manifolds. 
They have become a central area in geometry (cf. \cite{BR13}, \cite{CSZ12}, \cite{CC12}, \cite{CC13}, \cite{CCCMM14}, \cite{CD08}, \cite{CMM12}, \cite{CMMR17}, \cite{CMM16}, \cite{CMM17}, \cite{Chowetal07}, \cite{CLN06}, \cite{DS13}, \cite{Hamilton89}, \cite{HL14}, \cite{Maeta21}, \cite{Maeta23}, \cite{Maeta24}, \cite{Perelman02}). 
There are various types of geometric solitons, including Ricci solitons, Yamabe solitons, $k$-Yamabe solitons, quasi-Yamabe solitons, almost Yamabe solitons, conformal solitons. 
To provide a unified viewpoint for all these solitons, Catino, Mastrolia, Monticelli and Rigoli \cite{CMMR17} introduced Einstein-type manifolds as part of an ambitious project.

Let $(M,g)$ be a Riemannian manifold. If there exist $X\in\mathfrak{X}(M)$ and $\lambda\in C^\infty(M)$ such that 
\begin{equation}\label{ETM}
\alpha \Ric +\frac{\beta}{2}L_Xg+\mu X^\flat\otimes X^\flat =(\rho R+\lambda)g,
\end{equation}
for some constants $\alpha,\beta,\mu,\rho\in\mathbb{R}$, with $(\alpha,\beta,\mu)\not=(0,0,0)$, then $(M,g)$ is called an {\em Einstein-type manifold}.
Here, $\Ric$ denotes the Ricci tensor of $M$, $L$ denotes the Lie derivative, and $R$ is the scalar curvature of $M$.
If $X=\nabla F$ for some $F\in C^\infty(M)$, it is called a {\em gradient} Einstein-type manifold:
\begin{equation}\label{getm}
\alpha \Ric +\beta \nabla \nabla F+\mu \nabla F \nabla F =(\rho R+\lambda)g,
\end{equation}
where $\nabla\nabla F$ is the Hessian of $F$, and $F$ is called the potential function.
If $F$ is constant, then it is called {\em trivial}.

As pointed out in \cite{CMMR17}, gradient Einstein-type manifolds encompass many known soliton notions (cf. \cite{Hamilton82}, \cite{BR13}, \cite{CMM12}, \cite{HL14}):

$\bullet$ gradient Yamabe solitons: $(\alpha,\beta,\mu,\rho)=(0,1,0,1)$ and $\lambda$ is a constant.

$\bullet$ gradient almost Yamabe solitons: $(\alpha,\beta,\mu,\rho)=(0,1,0,1)$.

$\bullet$ gradient $k$-Yamabe solitons: $(\alpha,\beta,\mu,\rho)=\Big(0,\frac{1}{2(n-1)},0,0\Big)$ and $\lambda=\sigma_k-\nu$, where $\sigma_k$ is the $\sigma_k$-curvature and $\nu$ is a constant.

$\bullet$ gradient conformal solitons: $(\alpha,\beta,\mu,\rho)=(0,1,0,0)$.

$\bullet$ gradient quasi-Yamabe solitons: $(\alpha,\beta,\mu,\rho)=(0,1,-\frac{1}{k},1)$, where $k\in\mathbb{R}\setminus\{0\}$ and $\lambda$ is constant.

In this paper, we consider all the above cases, that is, gradient Einstein-type manifolds with $\alpha=0$. 
In \cite{CMMR17}, Catino, Mastrolia, Monticelli and Rigoli provided the structure theorem for Einstein-type manifolds with $\alpha=0$ (Theorem 1.4 in \cite{CMMR17}). 
Theorem 1.4 in \cite{CMMR17} is a generalization of theorems included in \cite{CSZ12}, \cite{CMM12}, \cite{Tashiro65}. 

As is well known, K\"ahler manifolds are central objects in both geometry and physics.
In particular, Calabi-Yau manifolds play a central role in the approach to elementary particle theory known as superstring theory (cf. \cite{LeeCM}). To understand the structure of K\"ahler manifolds, geometric flows and their self-similar solutions are considered on K\"ahler manifolds (cf. \cite{BEG13}, \cite{FIK03}, \cite{MW15}, \cite{PSB24}). 
Therefore, in this paper, we consider the broader notion of Einstein-type K\"ahler manifolds with $\alpha=0$, and completely classify non-trivial, complete gradient Einstein-type K\"ahler manifolds with $\alpha = 0$.

\begin{theorem}\label{main}
Any $2m$-dimensional $(m>1)$ non-trivial, complete, gradient Einstein-type K\"ahler manifold $(M,g,F)$ with $\alpha = 0$ is one of the following. 

\noindent
(I) When $\beta$ is zero, there exists no non-trivial Einstein-type K\"ahler manifold.

\noindent
(II) When $\mu$ is zero, the Einstein-type K\"ahler manifold is  either,\\
(A) a direct product $M=\mathbb{R}\times N^{2m-1}, g=dr^2+a^2g_N$, where $g_N$ is a Riemannian metric of $N$, and the potential function $F$ is $F(r)=ar+b$ for some constants $a>0$ and $b$, or\\
(B) the complex Euclidean space
and the potential function $F$ is $F(r)=ar^2+b$ for some constants $a>0$ and $b$.

\noindent
(III) When $\beta$ and $\mu$ have opposite signs, there exists no non-trivial Einstein-type K\"ahler manifold.

\noindent
(IV) When $\beta$ and $\mu$ have the same sign, the Einstein-type K\"ahler manifold is 
$M=[0,+\infty)\times \mathbb{S}^{2m-1}$ and $g=dr^2+\psi^2(r)g_S,$ $r\in[0,+\infty)$,
where $\psi(r)=F'(r)e^{-cF(r)}$. The potential function $F$ is
\[
F(r)=-\frac{1}{c}\log\left\{
c
\Big(
-\frac{a}{2}r^2-c_1
\Big)
\right\},
\]
for some positive constants $a$ and $c_1$, where $c=-\frac{\mu}{\beta}.$

\end{theorem}

Since $\lambda$ is a smooth function, the equation $\eqref{getm}$ can be reduced to
\begin{equation}\label{getm0}
\beta \nabla \nabla F+\mu \nabla F \nabla F =\varphi g,
\end{equation}
for $\varphi\in C^\infty(M)$ and constants $(\beta,\mu)\not=(0,0).$
Because of the symmetry of the sign of $(\beta,\mu)$, we only need to consider the following four cases.
\[
(\beta,\mu)=(0,P),(P,P),(P,0),(P,N),
\]
where $P$ and $N$ denote positive and negative constants, respectively.

We show the following lemma used in the proof of Theorem \ref{main}.
\begin{lemma}\label{keylem}
Let $(M,g)$ be an $n$-dimensional gradient Einstein-type manifold with $\alpha=0$. 
For the case $\beta\not=0,$ the following formulas hold.
\begin{align}\label{key1}
(n-1)\nabla_i \varphi-(n-1)c\varphi \nabla _iF+R_{ij}\nabla _jF=0,
\end{align}
\begin{align}\label{key2}
&(n-1)\nabla_\ell\nabla_i \varphi 
-(n-1)c\nabla_\ell\varphi\nabla _iF\\
-&(n-1)c\varphi^2 g_{\ell i}
-(n-1)c^2\varphi \nabla _{\ell} F\nabla _iF\notag\\
+&\nabla_\ell R_{ij}\nabla _jF
+\varphi R_{i\ell}
+cR_{ij}\nabla_\ell F\nabla_jF=0,\notag
\end{align}
\begin{align}\label{key3}
&(n-1)\Delta \varphi 
-(n-1)c\langle \nabla\varphi, \nabla F\rangle \\
-&(n-1)nc\varphi^2
-(n-1)c^2\varphi |\nabla F|^2\notag\\
+&\frac{1}{2}\langle \nabla R, \nabla F\rangle 
+\varphi R
+c\Ric(\nabla F,\nabla F)=0,\notag
\end{align}
where $c=-\frac{\mu}{\beta}$, and we redefine $\frac{\varphi}{\beta}$ as $\varphi$.
\end{lemma}

\begin{proof}
The gradient Einstein-type equation \eqref{getm0} can be denoted as follows.
\begin{equation}\label{getm00}
\nabla\nabla F=\varphi g+c\nabla F\nabla F,
\end{equation}
where $c=-\frac{\mu}{\beta}$, and we redefine $\frac{\varphi}{\beta}$ as $\varphi$.
Ricci identity shows 
\begin{equation}\label{1}
\Delta \nabla _iF=\nabla _i\Delta F+R_{ij}\nabla _jF.
\end{equation}
By the equation of Einstein-type manifolds, one has
\begin{align}\label{2}
\Delta\nabla _i F
=&\nabla_k\nabla_k\nabla_iF\\
=&\nabla_k(\varphi g_{ki}+c\nabla_kF\nabla_iF)\notag\\
=&\nabla _i\varphi+c\Delta F\nabla_iF+c\nabla_kF\nabla_k\nabla_iF\notag\\
=&\nabla _i\varphi+c\Delta F\nabla_iF+c\varphi\nabla_iF+c^2|\nabla F|^2\nabla _iF,\notag
\end{align}
and
\begin{align}\label{3}
\nabla_i\Delta F
=&\nabla_i (n\varphi+c|\nabla F|^2)\\
=& n\nabla _i\varphi+2c\nabla _i\nabla_kF\nabla_kF\notag\\
=& n\nabla _i\varphi+2c\varphi\nabla _iF+2c^2|\nabla F|^2\nabla_iF.
\notag
\end{align}
Combining \eqref{1} with \eqref{2}, one has
\[
\nabla_i\varphi+c\Delta F\nabla_iF+c\varphi\nabla_iF+c^2|\nabla F|^2\nabla_i F
-\nabla _i\Delta F-R_{ij}\nabla_j F=0.
\]
Substituting \eqref{3} into the equation, we have
\[
(n-1)\nabla_i \varphi-(n-1)c\varphi \nabla _iF+R_{ij}\nabla _jF=0.
\]
By differentiating \eqref{key1}, we have
\begin{align*}
&(n-1)\nabla_\ell\nabla_i \varphi
-(n-1)c\nabla_\ell\varphi \nabla _iF\\
-&(n-1)c\varphi \nabla_\ell\nabla _iF
+\nabla_\ell R_{ij}\nabla _jF
R_{ij}\nabla_\ell F\nabla_j F=0.
\end{align*}
Combining this with \eqref{getm00}, we have
\begin{align*}
&(n-1)\nabla_\ell\nabla_i \varphi 
-(n-1)c\nabla_\ell\varphi\nabla _iF\\
-&(n-1)c\varphi^2 g_{\ell i}
-(n-1)c^2\varphi \nabla _{\ell} F\nabla _iF\notag\\
+&\nabla_\ell R_{ij}\nabla _jF
+\varphi R_{i\ell}
+cR_{ij}\nabla_\ell\nabla_jF=0.\notag
\end{align*}
By taking trace,
\begin{align*}
&(n-1)\Delta \varphi 
-(n-1)c\langle \nabla\varphi, \nabla F\rangle \\
-&(n-1)nc\varphi^2
-(n-1)c^2\varphi |\nabla F|^2\notag\\
+&\frac{1}{2}\langle \nabla R, \nabla F\rangle 
+\varphi R
+c\Ric(\nabla F,\nabla F)=0.
\end{align*}

\end{proof}

\begin{proof}(Proof of Theorem \ref{main})

By Theorem 1.4 in \cite{CMMR17} (see also \cite{CSZ12} and the proof of Theorem 5.7.4 in \cite{PetersenRG}), $(M,g)$ is a complete warped product metric, and must have one of the three forms:

(1) $M=\mathbb{R}\times N$ and $g=dr^2+\psi^2(r)g_N$, and $F$ has no critical point.

(2) $M=[0,+\infty)\times \mathbb{S}^{2m-1}$ and $g=dr^2+\psi^2(r)g_S,$ $r\in[0,+\infty)$, and $F$ has only one critical point $F'(0)=0$.

(3) $M$ is compact and rotationally symmetric,\\
where $\psi(r)$ is a positive smooth function on $M$ depends only on $r$, and $g_S$ denotes the round metric of $\mathbb{S}^{2m-1}$.

\noindent
Case 1 $(\beta,\mu)=(0,P)$:

In this case, $\eqref{getm0}$ is
\[
\nabla F \nabla F=\varphi g,
\]
where we redefine $\frac{\varphi}{\mu}$ as $\varphi$.
By taking trace, one has
$|\nabla F|^2=2m\varphi.$
Since $F$ depend only on $r$, one has
\[
F'^2=2m\varphi.
\]
By the soliton equation, 
$F'^2=\varphi.$
Therefore, one has 
\[
(2m-1)F'^2=0.
\]
If $m>1$, then $F$ is constant.

We will consider the other cases. In these cases, $\beta\not=0$.
By the proof of Lemma \ref{keylem}, the equation of Einstein-type manifolds is denoted as \eqref{getm00}:
\[
\nabla\nabla F=\varphi g+c\nabla F\nabla F,
\]
where $c=-\frac{\mu}{\beta}$, and we redefine $\frac{\varphi}{\beta}$ as $\varphi$.
In this case, an elementary argument shows that 
\[
\psi(r)=F'(r)e^{-cF(r)}, 
\]
and $F'$ and $\psi$ are positive (see \cite{CMMR17}, \cite{CSZ12}, \cite{Maeta21}).
By the equation, for smooth vector fields $X,Y$ on $M$,
\[
g(\nabla _X\nabla F,Y)=\nabla\nabla F(X,Y)=g(\varphi X+cg(X,\nabla F)\nabla F,Y).
\]
Thus, one has
\[
\nabla_X\nabla F=\varphi X+cg(X,\nabla F)\nabla F.
\]
For any vector fields X,Y on $M$, a long direct computation shows that 
\begin{align*}
Rm(X,Y)\nabla F
=&\nabla _X\nabla_Y\nabla F-\nabla_Y\nabla_X\nabla F-\nabla_{[X,Y]}\nabla F\\
=&(X\varphi -c\varphi XF)Y-(Y\varphi -c\varphi YF)X.
\end{align*}
Since $J\nabla =\nabla J$, one has
\[
Rm(X,Y)J\nabla F
=J(Rm(X,Y)\nabla F)
=(X\varphi -c\varphi XF)JY-(Y\varphi -c\varphi YF)JX.
\]
Taking trace, we have 
\[
\Ric (Y,J\nabla F)=g(Y,-J(\nabla \varphi-c\varphi\nabla F)).
\]
Take $Y=JX$. Since the Ricci tensor is invariant under $J$,
\begin{align*}
\Ric(X,\nabla F)=\Ric (JX,J\nabla F)
=&g(JX,-J(\nabla\varphi-c\varphi\nabla F))\\
=&g(X,-(\nabla\varphi-c\varphi\nabla F)).
\end{align*}
By \eqref{key1}, we also have
\begin{align*}
\Ric(X,\nabla F)=g(X,-(2m-1)(\nabla \varphi-c\varphi \nabla F)).
\end{align*}
Combining these equations, one has,
\[
2(m-1)(\nabla \varphi-c\varphi\nabla F)=0.
\]
Since $m>1$, we have
\begin{equation}\label{key4}
\nabla \varphi=c\varphi\nabla F.
\end{equation}

Since $M$ is K\"ahler, if $M$ is compact, then it is trivial. 
Therefore, we only need to consider the other cases.
 
Since $F$ depends only on $r$, the equation of Einstein-type manifolds is as follows.
\[
\varphi=F''-cF'^2.
\]
By \eqref{key4}, we have the following ODE.
\begin{equation}\label{keyeq}
F'''-3cF'F''+c^2F'^3=0.
\end{equation}
Note that $F'$ does not take 0 on $\mathbb{R}$ in (1) and on $(0,+\infty)$ in (2).

\noindent
Case 2 $(\beta,\mu)=(P,0)$: 

Since $c=0$, $F'''=0.$ 
Therefore, $F'(r)=ar+b$ for some constants $a$ and $b$.

\noindent
(a) We consider case (1). If $a\not=0$, this case cannot occur. If $a=0$, then $F'$ is positive constant, and $F(r)=br+d$ for some constant $d$. 
Note that $\psi(r)=F'(r)$ (cf. \cite{CMM12}, \cite{Tashiro65}, \cite{Maeta21}).

\noindent
(b) We consider case (2). Since $F'(0)=0$, one has
$F'(r)=ar$, and $(M,g)$ is the complex Euclidean space $\mathbb{C}^n$.

We will consider the other cases.
Since $\psi=F'e^{-cF}$,
we have
\[
\psi'(r)=(F''-cF'^2)e^{-cF},
\]
and 
\[
\psi''(r) =(F'''-3cF'F''+c^2F'^3)e^{-cF}.
\]
Combining this with \eqref{keyeq}, we have $\psi''(r)=0$.
Thus, one has $\psi(r)=ar+b$.


Assume that $a=0$. Then, $F'e^{-cF}=\psi=b>0$.
Hence, $F'=be^{cF}$, and we have
\[
F(r)=-\frac{1}{c}\log 
\left(
c(-br-c_1)
\right),
\]
where $c_1$ is a constant.

\noindent
Case (1). Since $F'(r)=\frac{b}{c(-br-c_1)}$ must be smooth on $\mathbb{R}$, we have a contradiction.

\noindent
Case (2). Since $F'(0)=0$, we have a contradiction.

Assume that $a\not=0$. 
Then, $F'(r)e^{-cF(r)}=\psi(r)=ar+b$ for some constants $a$ and $b$.

\noindent
Case (1). Since $\psi>0$ on $\mathbb{R}$, we have a contradiction.

\noindent 
Case (2). Since $F'(0)=0$, we have $0=\psi(0)=b$. Hence, $\psi(r)=ar,~(a>0)$.
Thus, one has 
\[
F'(r)=are^{cF(r)}.
\]
Therefore, we have
\[
F(r)=-\frac{1}{c}\log
\left\{
c
\left(
-\frac{a}{2}r^2-c_1
\right)
\right\},
\]
where $c_1$ is a constant. Differentiating $F$, we have 
\[
F'(r)=\frac{ar}{c
\left(
-\frac{a}{2}r^2-c_1
\right).
}
\]
Since $F'$ is a positive smooth function on $(0,+\infty)$, it follows that $c<0$ and $c_1>0$ must hold.

\end{proof}

In Case 1 of the proof of Theorem \ref{main}, we did not use the assumption that $M$ is K\"ahler. Therefore, we have the following.

\begin{corollary}
Any $n$-dimensional $(n\geq3)$ non-trivial, complete, gradient Einstein-type manifold with $\alpha =\beta= 0$ is trivial. 
\end{corollary}

As exemplified by Perelman's conjecture \cite{Perelman02}, rotational symmetry often becomes a central problem for geometric solitons, particularly under certain curvature assumptions.
By applying Theorem \ref{main}, we obtain rotational symmetry for a wide class of geometric solitons. 

\begin{corollary}
Let $(M,g)$ be a $2m$-dimensional $(m>1)$ complete K\"ahler manifold that satisfies one of the following. If the potential function is non constant (that is, $M$ is nontrivial), then it is rotationally symmetric.
\begin{enumerate}
\item
A gradient Yamabe soliton with $R\not=-\lambda$ at some point.
\item
A gradient almost Yamabe soliton with $R\not=-\lambda$ at some point.
\item
A gradient $k$-Yamabe soliton with $\sigma_k\not=\nu$ at some point.
\item
A gradient conformal soliton with $\lambda\not=0$ at some point.
\item
A gradient quasi-Yamabe soliton.
\end{enumerate}
\end{corollary}

In particular, any nontrivial complete gradient quasi-Yamabe soliton on K\"ahler manifolds is rotationally symmetric, without any curvature assumption.


\bibliographystyle{amsbook}

\begin{thebibliography}{Chowetal07}  
\bibitem[BR13]{BR13}
E. Barbosa and  E. Ribeiro,
  {\em On conformal solutions of the Yamabe flow}, Arch. Math. {\bf 101} (2013), 79-89.


\bibitem[BEG13]{BEG13}
S. Boucksom, P. Eyssidieux, V. Guedj, 
{\it An introduction to the K\"ahler-Ricci flow},
Lecture Notes in Mathematics, {\bf 2086}, Springer, Cham, (2013).

\bibitem[Brendle05]{Brendle05}
S. Brendle,
{\it Convergence of the Yamabe flow for arbitrary initial energy}, 
J. Differential Geom., (2005), {\bf 69}, 217--278.

\bibitem[Brendle07]{Brendle07}
S. Brendle, 
{\it Convergence of the Yamabe flow in dimension 6 and higher}, 
Invent. Math., (2007), {\bf 170}, 541--576.







\bibitem[CSZ12]{CSZ12}
H.-D. Cao, X. Sun and Y. Zhang, 
{\it On the structure of gradient Yamabe solitons,}
Math. Res. Lett., (2012) {\bf 19}, 767--774.

\bibitem[CC12]{CC12}
H.-D. Cao and Q. Chen,
{\it On locally conformally flat gradient steady Ricci solitons},
Trans. Amer. Math. Soc., (2012), {\bf 364}, 2377-2391.

\bibitem[CC13]{CC13}
H.-D. Cao and Q. Chen, 
{\em On Bach-flat gradient shrinking Ricci solitons},
Duke Math. J., (2013) {\bf 162}, 1149--1169.

\bibitem[CCCMM14]{CCCMM14}
H.-D. Cao, G. Catino, Q. Chen, C. Mantegazza and L. Mazzieri, 
{\em Bach-flat gradient steady Ricci solitons},
Calc. Var.,  (2014) {\bf 49}, 125--138.


\bibitem[CMM12]{CMM12}
G. Catino, C. Mantegazza and L. Mazzieri, 
{\it On the global structure of conformal gradient solitons with nonnegative Ricci tensor},
Commun. Contemp. Math., (2012) {\bf 14}, 12pp.


\bibitem[CMM16]{CMM16}
G. Catino, P. Mastrolia and D.D. Monticelli,
{\it Classification of expanding and steady Ricci solitons with integral curvature decay},
Geom. Topol., (2016) {\bf 20}, 2665-2685.

\bibitem[CMM17]{CMM17}
G. Catino, P. Mastrolia and D.D. Monticelli,
{\it Gradient Ricci solitons with vanishing conditions on Weyl},
J. Math. Pure Appl., (2017) {\bf 108}, 1--13.

\bibitem[CMMR17]{CMMR17}
G. Catino, P. Mastrolia, D. Monticelli, M. Rigoli, 
{\it On the geometry of gradient Einstein-type manifolds,}
 Pac. J. Math. {\bf 286} (1) (2017) 39--67.
 
\bibitem[CD08]{CD08}
L. F. D. Cerbo and M. M. Disconzi,
{\it Yamabe Solitons, Determinant of the Laplacian and the Uniformization Theorem for Riemann Surfaces},
Lett. Math. Phys., {\bf 83} (2008), 13--18.

\bibitem[Chow92]{Chow92}
B. Chow, 
{\it The Yamabe flow on locally conformally flat manifolds with positive Ricci curvature}, Comm. Pure Appl. Math., (1992) {\bf 45 }, 1003-1014.

\bibitem[Chowetal07]{Chowetal07}
B. Chow, S.-C. Chu, D. Glickenstein, C. Guenther, J. Isenberg, T. Ivey, D. Knopf, P. Lu, F. Luo and L. Ni,
{\it The Ricci Flow: Techniques and Applications: Part I: Geometric Aspects},
Math. Surv. and Mono., Amer. Math. Soc., (2007), {\bf 135}.

\bibitem[CLN06]{CLN06}
B. Chow, P. Lu and L. Ni,
{\em Hamilton's Ricci Flow},
Graduate Studies in Mathematics, {\bf 77},  Amer. Math. Soc., (2006). 




\bibitem[DS13]{DS13}
P. Daskalopoulos and N. Sesum, 
\textit{The classification of locally conformally flat Yamabe solitons,}
Adv. Math., (2013) {\bf 240}, 346--369.

\bibitem[FIK03]{FIK03}
M. Feldman, T. Ilmanen and D. Knopf,
{\it Rotationally symmetric shrinking and expanding gradient K\"ahler-Ricci solitons,}
J. Differential Geom. {\bf65} (2003), no. 2, 169-209.

\bibitem[Hamilton82]{Hamilton82}
R. Hamilton, 
{\it Three-manifolds with positive Ricci curvature}, 
J. Differential Geom.,  (1982) {\bf 17}, 255--306.

\bibitem[Hamilton89]{Hamilton89}
R. Hamilton, 
{\it Lectures on geometric flows},
(1989), unpublished.

\bibitem[HL14]{HL14}
G. Huang and H.Li,
{\it On a classification of the quasi Yamabe gradient solitons,}
Methods Appl. Anal. {\bf 21}:3 (2014), 379--389.

\bibitem[LeeCM]{LeeCM}
J. M. Lee
{\it Introduction to complex manifolds},
Grad. Stud. Math., {\bf 244} American Mathematical Society, Providence, RI, 2024, xii+361 pp.

\bibitem[Maeta21]{Maeta21}
S. Maeta,
{\it Classification of generalized Yamabe solitons,}
arXiv:2107.05487[math DG].

\bibitem[Maeta23]{Maeta23}
S. Maeta,
{\it Complete steady gradient Yamabe solitons with positive scalar curvature are rotationally symmetric},
arXiv:2309.09166[math DG].

\bibitem[Maeta24]{Maeta24}
S. Maeta,
{\it Classification of low-dimensional complete gradient Yamabe solitons},
arXiv:2405.03921[math DG].

\bibitem[MW15]{MW15}
O. Munteanu and J. Wang, 
{\it Topology of K\"ahler Ricci solitons},
J. Differential Geom. {\bf 100} (2015), no. 1, 109-128.

\bibitem[Perelman02]{Perelman02}
G. Perelman, 
{\it The entropy formula for the Ricci flow and its geometric applications,}
arXiv math.DG/0211159, (2002).

\bibitem[PetersenRG]{PetersenRG}
P. Petersen,
{\it Riemannian Geometry, Third edition,}
Graduate Texts in Mathematics, (2016), {\bf 171}, Springer.

\bibitem[PSB24]{PSB24}
R. Poddar, R. Sharma, and S. Balasubramanian
{\it Remarks on real and complex Yamabe solitons},
J. Geom. (2024) 115:23.

\bibitem[SS03]{SS03}
H. Schwetlick and M. Struwe, 
{\it Convergence of the Yamabe flow for large energies,}
 J. Reine Angew. Math., (2003), {\bf 562}, 59-100.

\bibitem[Tashiro65]{Tashiro65}
Y. Tashiro, {\it Complete Riemannian manifolds and some vector fields},
 Trans. Amer. math. Soc., (1965), {\bf 117}, 251--275.


\bibitem[Ye94]{Ye94}
R. Ye,
{\it Global existence and convergence of Yamabe flow}
J. Differential Geom., (1994) {\bf 39}, 35-50.

\end{thebibliography}

\end{document}